\documentclass[12pt]{amsart}

\usepackage {crimson}

  \usepackage{amsfonts,graphics,amsmath,amsthm,amscd, amssymb,amsmath,latexsym,bbm,bm,mathtools,mathrsfs}
  \usepackage{flafter}
  \usepackage{verbatim,lipsum}
  \usepackage{enumerate}
  \usepackage{cases}

  \usepackage{url}
  \usepackage[pdftex,
              pdfauthor={Bronson Lim, Franco Rota},
              pdftitle={Characteristic classes and stability conditions for projective Kleinian orbisurfaces},
              pdfsubject={Algebraic geometry, stability condition, Kleinian singularity, stacky surfaces},
              pdfkeywords={stability stack bogomolov franco rota bronson lim},
              colorlinks=true]{hyperref}

  \usepackage[english]{babel}
  \usepackage[utf8]{inputenc}
\usepackage[top=30mm,bottom=30mm,left=30mm,right=30mm]{geometry}
  \usepackage{tikz}
  \usetikzlibrary{cd,matrix,arrows,decorations.pathmorphing}
  
\theoremstyle{plain}
\newtheorem{thm}{Theorem}[section]
\newtheorem{corollary}[thm]{Corollary}
\newtheorem{lemma}[thm]{Lemma}

\newtheorem{proposition}[thm]{Proposition}

\newtheorem*{thm*}{Theorem}
\newtheorem*{corollary*}{Corollary}
\newtheorem*{lemma*}{Lemma}
\newtheorem*{ld*}{Lemma/Definition}
\newtheorem*{proposition*}{Proposition}
\newtheorem*{assumption*}{Assumption}

\theoremstyle{definition}
\newtheorem{definition}[thm]{Definition}
\newtheorem{example}[thm]{Example}

\newtheorem*{definition*}{Definition}
\newtheorem*{example*}{Example}
\newtheorem*{xca*}{Exercise}
\newtheorem*{claim*}{Claim}
\newtheorem*{fact*}{Fact}
\newtheorem*{notation*}{Notation}
\newtheorem*{construction*}{Construction}
\newtheorem*{ack*}{Acknowledgements}
\newtheorem*{question*}{Question}
\newtheorem*{problem*}{Problem}
\newtheorem*{conjecture*}{Conjecture}

\theoremstyle{remark}
\newtheorem{remark}[thm]{Remark}

\newcommand{\C}{\mathbb{C}}

\newcommand{\Q}{\mathbb{Q}}
\newcommand{\R}{\mathbb{R}}

\DeclareMathOperator{\rk}{rk}
\DeclareMathOperator{\Tr}{Tr}

\DeclareMathOperator{\SL}{SL}

\DeclareMathOperator{\Hom}{Hom}

\DeclareMathOperator{\re}{Re}
\DeclareMathOperator{\im}{Im}

\DeclareMathOperator{\dL}{\mathbf{L}}
\DeclareMathOperator{\dR}{\mathbf{R}}

\DeclareMathOperator{\Stab}{Stab}

\DeclareMathOperator{\textch}{ch}												%chern classes

\newcommand{\ch}[1]{\textch_{{#1}}}
\newcommand{\bfv}{\mathbf{v}}
\newcommand{\bfw}{\mathbf{w}}

\DeclareMathOperator{\zPer}{\leftidx{^0}{Per}{}}

\DeclareMathOperator{\Coh}{Coh}

\newcommand{\sA}{\mathcal{A}}
\newcommand{\sB}{\mathcal{B}}

\newcommand{\sO}{\mathcal{O}}

\newcommand{\sS}{\mathcal{S}}
\newcommand{\sT}{\mathcal{T}}
\newcommand{\sM}{\mathcal{M}}
\newcommand{\sC}{\mathcal{C}}
\newcommand{\sD}{\mathcal{D}}

\newcommand{\sE}{\mathcal{E}}
\newcommand{\sF}{\mathcal{F}}

\newcommand{\PP}{\mathcal{P}}

\newcommand{\st}{\,|\,}					%such that
\usepackage{leftidx}						%allows sub/superscripts on the left of the symbol

\newcommand{\AAA}[1]{\textcolor{red}{#1}}

\DeclarePairedDelimiter{\set}{\lbrace}{\rbrace}

\DeclarePairedDelimiter{\pair}{\langle}{\rangle}
\DeclarePairedDelimiter{\norm}{\lVert}{\rVert}
\DeclarePairedDelimiter{\abs}{\lvert}{\rvert}

\begin{document}

\title{Characteristic classes and stability
 conditions for projective Kleinian orbisurfaces}
%how about: Characteristic classes and stability conditions for kleinian orbisurfaces

\author[B. Lim]{Bronson Lim}
\address{BL: Department of Mathematics \\ University of Utah \\ Salt Lake City,
  UT 84102, USA} 
\email{bronson@math.utah.edu}

\author[F. Rota]{Franco Rota}
\address{FR: Department of Mathematics, Rutgers University, Piscataway, NJ 08854, USA}
%\curraddr{Department of Mathematics and Statistics,
%Case Western Reserve University, Cleveland, Ohio 43403}
\email{rota@math.rutgers.edu}

%\thanks{2010 {\it Mathematics Subject Classification.} 18E30; 14H45, 14J33
%}

% 18E30     derived,triangulated categories
% 14H45     special curves and curves of low genus
% 14J33     mirror symmetry
% 14A20     Generalizations (algebraic spaces, stacks)
% 14J17  	Singularities
% 14F05  	Sheaves, derived categories of sheaves and related constructions
% 14E16  	McKay correspondence
% 14C40  	Riemann-Roch theorems

%    General info
%\subjclass[2000]{Primary 54C40, 14E20; Secondary 46E25, 20C20}

%\date{January 1, 2001 and, in revised form, June 22, 2001.}



\begin{abstract}
We construct Bridgeland stability conditions on the derived category of smooth quasi-projective Deligne--Mumford surfaces whose coarse moduli spaces have ADE singularities. This unifies the construction for smooth surfaces and Bridgeland's work on Kleinian singularities. The construction hinges on an orbifold version of the Bogomolov--Gieseker inequality for slope semistable sheaves on the stack, and makes use of the To\"en--Hirzebruch--Riemann--Roch theorem.
\end{abstract}

\maketitle

%\setcounter{tocdepth}{1}
%\tableofcontents

%\input{intro.tex}

%\mainmatter

\section{Introduction}

\subsection*{Bridgeland Stability}
Stability conditions were introduced by Bridgeland in \cite{Bri07_triang_cat}
following work of Douglas on $\Pi$-stability \cite{Dou02}. Since then, the
problem of constructing stability conditions has been investigated successfully
for triangulated categories $\sD$ coming from a variety of different sources. In
fact, stability conditions are completely classified if $\sD$ is the derived
category of a smooth curve (see \cite{Mac07} and references therein) and there
is a procedure to construct stability conditions on derived categories of smooth
projective surfaces (a first construction appears in  \cite{Bri08_k3}, and is
then generalized in \cite{AB13}. See also the survey \cite{MS17} for a thorough
account on the matter). On a related line of investigation, Bridgeland studies
the stability manifold of categories associated with the class of ADE surface
singularities \cite{Bri09_kleinian}. The most recent results on this matter
concern threefolds \cite{BMS16}, \cite{BMSZ17}, \cite{Li19}.

In this work, we extend the general construction for surfaces to
$\sD=D^b(\Coh(\sS))$, where $\sS$ is the canonical stack associated with a
projective surface $S$ with ADE singularities. The main result of the paper is
Theorem \ref{thm_OurStabilityCondition}: it unifies the construction for smooth
surfaces \cite{AB13} and Bridgeland's work on Kleinian singularities
\cite{Bri09_kleinian}. To prove Theorem \ref{thm_OurStabilityCondition} we
develop a strengthening of the Bogomolov--Gieseker inequality for slope
semistable sheaves on $\sS$, and make use of the
To\"en--Hirzebruch--Riemann--Roch theorem \cite{Toe99}.

\subsection*{Notation and conventions} Throughout, we work over the field of complex numbers. We denote by $\pi\colon\sS\to S$ the
canonical stack associated with a surface $S$ with isolated quotient
singularities, and by $f\colon\tilde S\to S$ its minimal resolution. For a
finite subgroup $G\subset \SL_2$, we denote by $\rho_0=\mathbbm 1$ its
trivial representation, and by $\rho_i$, $i=1,...,M$, its non-trivial ones. We
set $N\coloneqq \abs{G}$.

If $X$ is a smooth scheme or algebraic stack, $D(X)\coloneqq D^b(\Coh(X))$
denotes the bounded derived category of coherent sheaves on $X$.

\subsection*{Summary of results}

Let $\sS$ be the canonical stack associated to a projective surface $S$ with a unique Kleinian singularity, and denote by $\tilde{S}$ its minimal resolution. Let $\iota\colon BG\to \sS$ be the associated residual gerbe. If $H$ is an ample divisor on $\sS$, we define the slope of a sheaf $E$ on $\sS$: 
\[
    \mu(E) = \frac{H\cdot \ch 1(E)}{\mathrm{rk}(E)},
\]
and its discriminant $\Delta(E)\coloneqq \ch 1(E)^2-2\ch 0(E)\ch 2(E)$ (the Chern classes are the ones introduced in \cite{Vistoli}).
Tilting with respect to slope, we construct a heart of a bounded t-structure $\mathcal \Coh^{b}(\sS)$, for $b\in \R$. 
 
 To define a suitable central charge, we need to take into account the orbifold cohomology of $\sS$. We have $ [\dL\iota^*E]=\sum_{i=0}^M a_i\rho_i$, and define the \textit{orbifold Chern character} of $E$ as
 \[ 
    \ch{orb}(E)=\left(\ch{}(E), a_0,...,a_{M}\right).
  \] 
To\"en's version of the Riemann Roch theorem (Theorem \ref{thm:tgrr}) involves a function 
\begin{equation}
\label{eq_delta_Ti}
\delta(E)=\sum_ia_iT_i
\end{equation}
where the $T_i$ are rational coefficients that depend on $G$. For $w\in\C$ and $\gamma\in\R$, let $Z_{w,\gamma}$ be the function
\[ Z_{w,\gamma}(E)=Z(\ch{orb}(E))= -\ch 2(E) + w \ch 0(E) + \gamma.\delta(E) +iH.\ch 1 (E). \]

The main result of the paper is Theorem \ref{thm_OurStabilityCondition}:
\begin{thm}Let $N\coloneqq \abs{G}$, and choose parameters $\gamma\in (0,\frac{1}{N-1})$ and $w\in \C$ such that:
\begin{enumerate}[(i)]
    \item $\re w > - \frac{(\im w)^2}{H^2}+ (2+\gamma)D -  (1+\gamma)^2$;
    \item $\re w >\frac12 \frac{(\im w)^2}{H^2}  - \gamma(D-\frac{N-1}{N})>0$.
\end{enumerate} 
Then, the pair $(Z_{w,\gamma}, \Coh^{-\im w}(\sS))$ is a stability condition on $\sD$.
\end{thm}

The proof hinges on a Bogomolov--Gieseker type inequality involving the \textit{orbifold discriminant}: this is defined through the McKay equivalence $\Phi\colon \sD \xrightarrow{\sim} D^b(\Coh(\tilde{S}))$ \cite{BKR01}, as the form
\[\Delta_{orb}(E)\coloneqq \Delta(\Phi(E)).\] 

We obtain Theorem \ref{thm_discriminant}, which strengthens the usual Bogomolov--Gieseker inequality.
\begin{thm}
Let $ E $ be a $\mu_H$-semistable sheaf on $\sS$. Then, $\Delta_{orb}( E  )\geq 0$.
\end{thm}

The form $\Delta_{orb}$ also plays a crucial role in proving the support property, as it turns out to be negative definite on the kernel of $Z_{w,\gamma}$ (Lemma \ref{lem_Q0_neg_def}).

In Section \ref{sec_WallCrossing} we study wall-crossing for objects of class $[\sO_x]$, where $x\in \sS$ is a closed point with trivial stabilizer. As a result, we find stability conditions $\sigma$ and $\sigma_0$ such that the corresponding moduli spaces coincide with $S$ and $\tilde{S}$, respectively, and that the wall-crossing morphism
\[ M_{\sigma_0}([\sO_x]) \to M_{\sigma}([\sO_x]) \]
is exactly the minimal resolution of $S$ (Proposition \ref{prop_recoverGHilb}). Section \ref{sec_ComparisonBridgeland} further illustrates the relation between our construction, quiver stability \cite{Kin94}, and \cite{Bri09_kleinian}. 

%The appendix contains a computation of the coefficients $T_i$ appearing in \eqref{eq_delta_Ti} for all finite subgroups of $\mathrm{SL}_2$. Our computation completes that of \cite{Lie11} (which treats the cyclic case) and of \cite{CT19} (where only $T_0$ is computed).

\begin{remark}
In light of the equivalence $\sD \xrightarrow{\sim} D^b(\Coh(\tilde{S}))$, this work can be compared to some previous results on surfaces. In fact, it is closely related to \cite{Tra17}, where the authors construct stability conditions on smooth surfaces admitting a curve of negative self-intersection. Theorem \ref{thm_OurStabilityCondition} overlaps with \cite[Theorem 5.4]{Tra17} in the $A_2$ singularity case. 

The heart used in \cite{Tra17} is constructed by tilting coherent sheaves twice (this construction also appears in \cite{Tod13}), while working on the stack appears to be a more natural choice, as it only requires one tilt, in accordance with the expectation for surfaces.
\end{remark}

%Moduli spaces of semistable sheaves on a Deligne-Mumford stack are constructed in \cite{Nir08_moduli}. It would be interesting to compare them with spaces of $\sigma$-semistable objects, where $\sigma$ is one of the conditions of Theorem \ref{thm_OurStabilityCondition}. 

%While an analog of \cite{Bri09_kleinian} is missing in the case that $G$ is a subgroup of $\mathrm{GL}_2$, a version of the derived McKay correspondence still holds \cite{Wem11}. This work may suggest a future line of investigation for non-Gorenstein surface singularities. 

\subsection*{Acknowledgements} We wish to thank Aaron Bertram and Michael Wemyss for the many fruitful conversations on this topic.

%%%%%%%%%%%%%%%%%%%%%%%%%%%%%%%%%%%%%%%%%%%%%%%%%%%%%%%%%%
%%%%%%%%%%%%%%%%%%%%%%%%%%%%%%%%%%%%%%%%%%%%%%%%%%%%%%%%%%
%%%%%%%%%%%%%%%%%%%%%%%%%%%%%%%%%%%%%%%%%%%%%%%%%%%%%%%%%%
%%%%%%%%%%%%%%%%%%%%%%%%%%%%%%%%%%

\section{Preliminaries}

\subsection{Kleinian orbisurfaces}

\begin{definition}
  An \textit{orbisurface} is a smooth and proper Deligne-Mumford surface such
  that the stacky locus has codimension 2. 
  \label{def:orbisurface}
\end{definition}

For any orbisurface \(\sS\), and geometric point \(s\in \sS\), there is an
\'etale local chart near \(s\):
\[
  j_s\colon [U/\mathrm{st}(s)]\to \sS
\]
where \(U\subset\mathbb{A}^2\) is open and \(\mathrm{st}(s)\) is the stabilizer
group of \(s\) acting through \(\mathrm{GL}_2\). The mapping \(j_s\) induces a
closed embedding
\[
  j_s\colon [\ast/\mathrm{st}(s)]\to \sS
\]
called the \textit{residual gerbe} at \(s\). We denote by \(BG\) the quotient
stack \([\ast/G]\). 

An orbisurface is \textit{Kleinian} if for each \(s\in \sS\), the stabilizer
group acts through \(\mathrm{SL}_2\). And an orbisurface is an
\(A_{N-1}\)-orbisurface if it is Kleinian and the non-trivial
stabilizer groups are cyclic of order \(N\).

Let \(S\) be a surface with Kleinian singularities. Then there exists a Kleinian
orbisurface \(\sS\) and a map \(\pi\colon\sS\to S\) such that:
\begin{itemize}
\item the restriction $  \sS\setminus\pi^{-1}(\mathrm{Sing}(S))\to S\setminus\mathrm{Sing}(S)$ is an isomorphism;
\item $\pi$ is universal among all dominant, codimension
preserving maps to \(S\).
\end{itemize}
The stack \(\sS\) is called the \textit{canonical stack}
associated with the surface \(S\), see \cite{FMN10}.

A line bundle on \(\sS\) is \textit{ample} if it is the pullback of
an ample line bundle on the coarse space \(S\). An orbisurface
is \textit{projective} if the coarse moduli is projective.

%\begin{remark}
  %The isomorphism in Definition \ref{def:orbisurface} is local in the Zariski
  %topology (see Sec. 2.1 of \cite{Lieblich_PERIOD AND INDEX IN THE BRAUER GROUP
  %OF AN ARITHMETIC SURFACE}). \AAA{I don't see where this is said in the
  %referenced paper. And I have a hard time believing its true. For example, the
  %Kummer construction yields the stack \([A/\mu_2]\) which is Kleinian but does
  %not have a Zariski local chart of the form you mention.}
%\end{remark}

\begin{example}
  The weighted projective plane \(\mathbb{P}_{1,1,N}\) has canonical stack the
  stacky weighted projective plane
  \[
    \textbf{P}_{1,1,N} = [(\mathbb{C}^3_{1,1,N}\setminus\{0\})/\mathbb{C}_m]
  \]
  where the subscript indicates the weights of the \(\mathbb{C}^\ast\)-action. That
  is, \(\lambda\in\mathbb{C}^\ast\) acts by \(\lambda(x,y,z) = (\lambda x,\lambda
  y,\lambda^Nz)\). There is a unique stacky point where \(x\) and \(y\)
  are zero with residual gerbe \(B\mu_N\). Thus the stacky weighted
  projective plane is a projective \(A_{N-1}\)-orbisurface.
\end{example}

\begin{example}
  The local model for a surface with an $A_{N-1}$ singularity is the hypersurface
  \[
    S=\{x^2+y^2+z^N=0\}
  \]
  in \(\mathbb{C}^3\). The canonical stack is the \(A_{N-1}\)-orbisurface 
  \[
    \sS = [\mathbb{C}^2/\mu_N]
  \]
  where \(\lambda\in\mu_N\) acts via \( \lambda(u,v) = (\lambda
  u,\lambda^{-1}v)\). 
\end{example}

Although we are primarily interested in the case where there is a unique stacky
point, the following example should be kept in mind.

\begin{example}
  Let \(A\) be an Abelian surface and let \(\mu_2=\langle -1\rangle\) act on
  \(A\) via negation, i.e. \(-1\cdot a = -a\) for all \(a\in A\). There are
  sixteen fixed points of this action. Thus the quotient stack
  \([A/\mu_2]\) is an \(A_1\)-orbisurface with sixteen residual gerbes of type
  \(B\mu_2\).
\end{example}

\subsection{The derived McKay correspondence }\label{sec_McKay_corr}

A Kleinian orbisurface $\sS$ can be interpreted as a \emph{stacky resolution of
singularities} of its coarse moduli space $S$. The derived McKay correspondence
\cite{BKR01} exhibits an equivalence $\Phi$ between the derived category
$D(\sS)$ and that of the minimal resolution $f\colon \tilde{S}\to S$ of $S$.

Let $\tilde{\sC}$ be the abelian subcategory of $\Coh(\tilde{S})$ consisting of
sheaves $E$ such that $\dR f_*(E)=0$, and define a torsion pair:
\begin{align*}
  \tilde{\sT}_0&\coloneqq\set{T\in\Coh(\tilde{S})\st \dR^1f_*(T)=0};\\
  \tilde{\sF}_0&\coloneqq\set{F\in\Coh(\tilde{S})\st f_*(F)=0, \Hom(\tilde\sC,F)=0}.
\end{align*}

The heart of the bounded $t$-structure on $D(\tilde{S})$ obtained by tilting
$\Coh(\tilde{S})$ along the pair above is denoted $\zPer(\tilde{S}/S)$, its
objects are called \textit{perverse sheaves}. The reader is referred to
\cite{Bri02_flops} and \cite{VdB04} for the details on this construction.

The derived Mckay correspondence, $\Phi$, satisfies:
\[ 
  \Phi(\Coh(\sS))\simeq \pair{\tilde{\sF}_0[1],\tilde{\sT}_0} = \zPer(\tilde{S}/S). 
\]
More explicitly, suppose \(S\) has a unique singular point \(p\). Let \(\sS\) be
the associated canonical stack and, abusing notation, \(p\) the lift of the
point \(p\) to \(\sS\). Denote by $C$
the fundamental cycle of $\tilde{S}\to S$, and by $C_i$ its irreducible
components. Then we have
\begin{align*}
  \Phi(\sO_\sS) &=\sO_{\tilde{S}};\\
  \Phi(\sO_p) &= \omega_{C}[1];\\
  \Phi(\sO_p\otimes \rho_i) &= \sO_{C_{i}}(-1), \quad i=1,...,M.
\end{align*} 

We fix a quasi-inverse $\Phi^{-1}$ of $\Phi$ and write $\sF_0\coloneqq
\Phi^{-1}(\tilde{\sF}_0[1])$ and $\sT_0\coloneqq \Phi^{-1}(\tilde{\sT_0})$, so
that 
\[ 
  \Coh(\sS)=\pair{\sF_0,\sT_0}. 
\]
Moreover, the category $\sC$ of sheaves $E$ on $\sS$ such that $\dR\pi_*(E)=0$
satisfies $\sC=\Phi^{-1}\tilde\sC$, and is generated by the sheaves
$\sO_p\otimes\rho_i$, $i\neq 0$.

We finish this section by recalling a definition which will be useful later.
\begin{definition}\label{def_cluster_constellation}
  Let $W$ be a quasi-projective variety, on which a finite group $G$ is
  acting. A $G$-\textit{constellation} on $W$ is a $G$-equivariant sheaf $E$ on $W$ with finite
  support such that $H^0(E)$ is isomorphic to the regular representation of $G$,
  as $G$-representation. A $G$-\textit{cluster} is the structure sheaf $\sO_Z$
  of a subscheme $Z\subset W$ which is also a $G$-constellation.
\end{definition}

If \(G\) is a finite subgroup of
\(\mathrm{GL}_2\) acting on \(\mathbb{C}^2\), then the space of $G$-clusters, denoted
$G$-Hilb$(\C^2)$, is the minimal resolution of $\C^2/G$ \cite{BKR01}. The skyscraper sheaves of points in the exceptional locus correspond under $\Phi^{-1}$
to clusters supported at the origin in $\C^2$.

\subsection{Characteristic classes}
\label{sec_Chern_classes}

From now on, we assume that \(\sS\) is a projective
Kleinian orbisurface with a unique stacky point \(p\in \sS\) and
residual gerbe \(BG=[\ast/G]\). Set \(\iota\colon
BG\hookrightarrow \sS\) the corresponding closed substack. 

We use Vistoli's intersection theory in what follows \cite{Vistoli}. In
particular, Chern classes and Todd classes are defined, as well as a degree map.
The Hodge index theorem still holds, i.e. the intersection
form on \(\mathrm{NS}(\sS)\otimes\mathbb{R}\) is of signature \( (1,r-1)\):

\begin{thm}[Hodge Index Theorem]
    Suppose \(H\) is an ample Cartier divisor on \(\sS\). If \(D\not\equiv 0\) is a divisor
    such that \(D\cdot H=0\) then \(D^2<0\).
    \label{thm:hodge-index}
\end{thm}

For $E$ a sheaf, and $H$ an ample divisor class on $\sS$, the
slope of \(E\) with respect to \(H\) is
\[
    \mu(E) = \frac{H\cdot \ch 1(E)}{\mathrm{rk}(E)},
\]
with the convention that $\mu(E)=+\infty$ if $\mathrm{rk}(E)=0$. We say that
$E$ is $\mu$-(semi)stable if for all non-zero proper subsheaves $E'\subset E$
one has $\mu(E')<(\leq) \mu(E/E')$.

Define also the discriminant of \(E\) by
\[
    \Delta(E) =(\ch 1(E))^2 - 2\mathrm{rk}(E)\ch 2(E).
\]

The usual Bogomolov-Gieseker inequality still holds on $\sS$:

\begin{thm}[{\cite[Prop. 4.2.4]{Lie11}}]
    If \(E\) is a \(\mu\)-semistable sheaf, then $\Delta(E)\geq 0$, or equivalently
    \begin{equation}\label{eq:bg-inequality}
        \ch 2(E)\leq \frac{(\ch 1(E))^2}{\mathrm{rk}(E)}.
    \end{equation}
    \label{thm:bg-inequality}
\end{thm}

The results above only involve a part of the Grothendieck group of $\sS$, and ignore contributions from the residual gerbe $BG$. The Grothendieck group of \(BG\) is free, Abelian and generated by the
irreducible representations of $G$ $\{ \rho_i \,|\,i=0,...,M \}$. For any
perfect complex of sheaves \(E\) on \(\sS\), we have
\[
    [{\dL}\iota^\ast E] = \sum_{i=0}^{M} a_i\rho_i.
\]

\begin{definition}\label{def_orbifold_Chern_char}
  Given a perfect complex  \(E\in \sD(\sS)\), we define the \emph{orbifold Chern
  character} 
  \[ 
    \ch{orb}(E)=\left(\ch{}(E), a_0,...,a_{M}\right).
  \]
\end{definition}

\subsection{The To\"en--Hirzebruch--Riemann--Roch theorem} 

We use a version of the Hirzebruch-Riemann-Roch theorem for smooth projective
Deligne-Mumford stacks due to To\"en \cite{Toe99}. The formula is analogous to
the usual Hirzebruch-Riemann-Roch theorem, but it presents a correction term.
For the convenience of the reader, we give a brief description of the formula,
following \cite[Appendix A]{Tse10}.

Let $I\sS$ denote the inertia stack of $\sS$, and define a map
\[ 
  \rho\colon K(I\sS) \to K(I\sS)\otimes \Q(\mu_\infty) 
\]
as follows: if $E$ is a bundle on $I\sS$ decomposing as a sum
$\bigoplus\limits_{\zeta} E^{(\zeta)}$ of eigenbundles with eigenvalue $\zeta$,
let 
\[ 
  \rho(E)=\sum\limits_{\zeta} \zeta E^{(\zeta)}.
\]
One then defines the weighted Chern character as the composition
\[ 
  \widetilde{\ch{}}\colon K(\sS)\xrightarrow{\sigma^*} K(I\sS) \xrightarrow{\rho}
  K(I\sS) \xrightarrow{\ch{}} H^*(I\sS) 
\]
where $\sigma\colon I\sS \to \sS$ is the projection and $\ch{}$ is the usual Chern
character. The weighted Todd class $\widetilde{\mathrm{Td}}_\sS$ is defined in a
similar way \cite[Def. A.0.5]{Tse10}. Then we have

\begin{thm}[To\"en-Hirzebruch-Riemann-Roch]
    Let \(E\) be a perfect complex of sheaves on \(\sS\), then 
    \[
    \chi(E) = \int\limits_{I\sS}\widetilde{\ch{}}(E).\widetilde{\mathrm{Td}}_\sS =  \int\limits_{\sS}\widetilde{\ch{}}(E).\widetilde{\mathrm{Td}}_\sS + \delta(E)
        %case N=2    
        %\chi(\sF) = \deg(\ch{}(\sF).\mathrm{Td}_X)+\frac{a}{4}-\frac{b}{8}
    \]
    where $\delta(E)\coloneqq \int\limits_{I\sS\setminus
    \sS}\widetilde{\ch{}}(E).\widetilde{\mathrm{Td}}_\sS$ is the aforementioned
    correction term.
    \label{thm:tgrr}
\end{thm}

Our short term goal is now to investigate the term $\delta(E)$ in the case of
a Kleinian orbisurface, by computing the weighted Chern characters of
$[\dL\iota^*E]$. The inertia stack of $\sS$ is 
\[ 
  I\sS = \sS \sqcup (IBG \setminus BG) 
\]
where 
\[ 
  IBG \setminus BG = \bigsqcup\limits_{(g)\neq (1)}BC_G(g)
\]
(here, the union is taken over all conjugacy classes $(g)$ of non-trivial elements $g\in G$).
The degree of the Todd class on $IBG\setminus BG$ is given by the formula
\[ 
  \int\limits_{IBG\setminus BG} \widetilde{\mathrm{Td}}_\sS =
  \sum\limits_{(g)\neq (1)}\frac{1}{\abs{C_G(g)}} \cdot \frac{1}{2-\xi_g -
  \xi_g^{-1}} 
\]
where $\xi_g$ and $\xi_g^{-1}$ are the eigenvalues of the action of $g$ on the
tangent space $T_p\sS$ of the stacky point on $\sS$.  This number is computed in
\cite{CT19} to be 
\[
  \delta(\sO_\sS) = \frac{1}{12}\left(\chi_{top}(C_{red})-\frac{1}{\abs{G}}\right),
\]
where $C$ is the fundamental cycle of the minimal resolution (see Section \ref{sec_McKay_corr}). 

%\AAA{Why the notation \(C_{red}\)? Is there another \(C\) being used?}

The fiber of a sheaf $E$ at \(p\) decomposes as $[\dL\iota^*E]=\sum_{i=0}^M a_i
\rho_i$ where the sum runs over all irreducible representations $\rho_i$ of $G$.
On $BC_G(g)$, the element $g$ acts on $\rho_i$ with eigenvalues $\zeta_i^{(l)}$,
to which correspond eigenspaces $\rho_i^{(l)}$. Therefore, $\dL\iota^*E$
decomposes on $BC_G(g)$ into weighted eigenbundles as $\sum\limits_{i=0}^M
\sum\limits_{l=1}^{r_i} a_i \zeta_i^{(l)} \rho_i^{(l)}$.

Then, the weighted Chern character of $\dL\iota^*E_{|BC_G(g)}$ is given by 
\begin{equation}\label{eq_weighted_Chern_char}
  \widetilde{\ch{}}(\dL\iota^*E_{|BC_G(g)}) = \sum\limits_{i=0}^M \sum_{l=1}^{r_i} a_i \zeta_i^{(l)} =  \sum_{i=0}^M a_i \chi_i(g),
\end{equation}
where $\chi_i\coloneqq \chi_{\rho_i}= \Tr \circ \rho_i$ is the character of the
representation $\rho_i$. Our main interest lies in the following computation:

\begin{lemma}\label{lem_delta_of_skyscrapers}
Let $\rho$ be an irreducible representation of $G$ of dimension $r$. Then the second Chern character of $\sO_p\otimes \rho$ is $\frac{r}{N}$, and 
\begin{equation}
    \delta(\sO_p\otimes \rho)=\begin{cases}
    1-\frac{1}{N} \qquad \text{ if }\rho = \mathbbm 1;\\
    -\frac{r}{N} \qquad \text{ if }\rho\neq \mathbbm 1.
    \end{cases}
\end{equation}
\end{lemma}

\begin{proof}
  This is a local computation and so we can assume \(\sS = [U/G]\) where \(U\)
  is an open subset of \(\mathbb{A}^2\). In this case, we have the equivariant Koszul complex (write $V$ to denote $T_p\sS$ as a representation of $G$)
  \[ 
    0\to \sO_U\otimes \Lambda^2V\cong\sO_U \to \sO_U \otimes V \to \sO_U \to
    \sO_p,
  \]
  which resolves $\sO_p$. Hence,
  \[  
    [\dL\iota^*(\sO_p\otimes \rho)]=(2\cdot \mathbbm 1-V)\otimes \rho. 
  \]
  By Theorem \ref{thm:tgrr} and multiplicativity of characters, the correction term is 
  \begin{align*}
        \delta(\sO_p\otimes \rho) &=
        \sum\limits_{(g)\neq (I)}
        \frac{1}{\abs{C_G(g)}}\cdot\frac{\widetilde{\ch{}}(\dL\iota^*\sO_{p|BC_G(g)})}{2-\xi_g
        - \xi_g^{-1}} \\
        &= \sum\limits_{(g)\neq (I)}
        \frac{1}{\abs{C_G(g)}}\cdot\frac{(2\chi_{\mathbbm 1}(g) -
        \chi_V(g))\chi_\rho(g)}{2-\chi_V(g)}\\
        &= \sum\limits_{(g)\neq (I)} \frac{\chi_\rho(g)}{\abs{C_G(g)}} \\
  \end{align*}
  Denote by $N_g$ the cardinality of the conjugacy class of $g\in G$, and write the orthogonality relation between characters:
  \[ 
    \delta_{\mathbbm 1\rho}= \frac 1N\sum_{g\in
    G}\chi_\rho(g)\overline{\chi_{\mathbbm 1}(g)}=
    \frac{1}{N}\sum_{(g)}N_g\chi_\rho(g)=
    \frac{1}{N}\sum_{(g)}\frac{N}{\abs{C_G(g)}}\chi_\rho(g)=
    \sum_{(g)}\frac{\chi_\rho(g)}{\abs{C_G(g)}},  
  \]
  where $\delta_{\mathbbm 1\rho}=1$ if $\rho=\mathbbm 1$ and 0 otherwise. 

  The summand corresponding to $(g)=(I)$ is $\frac{\chi_\rho(I)}{N}=\frac rN$. Isolating it, one obtains
  \[  
    \sum\limits_{(g)\neq (I)} \frac{\chi_\rho(g)}{\abs{C_G(g)}} =\begin{cases}
    1-\frac{1}{N} \qquad \text{ if }\rho = \mathbbm 1;\\
    -\frac{r}{N} \qquad \text{ if }\rho\neq \mathbbm 1.
    \end{cases}
  \]
  Since
  $\chi(\sS,\sO_p\otimes\rho)=\chi(S,\pi_*(\sO_p\otimes\rho))=\delta_{\mathbbm
  1\rho}$, the statement about second Chern characters follows.

\end{proof}

 \section{A Bogmolov-Gieseker-type inequality for slope semistable sheaves}

Let $\sS$ be a projective Kleinian orbisurface, denote $\pi\colon \sS \to S$ the
structure morphism, and $f\colon \tilde S\to S$ the minimal resolution of $S$.
We keep our standing assumption that \(\sS\) has only one stacky point \(p\)
with residual gerbe \(BG\) and \(G\) acts through \(\mathrm{SL}_2\).

Recall that for a sheaf \(E\) on $\sS$ or $\tilde S$, and $H$ an ample (on $\sS$)
divisor class, the slope of \(E\) with respect to \(H\) is
\[
    \mu_H(E) = \frac{H\cdot \ch{1}(E)}{\ch 0(E)}.
\]
and its discriminant is $ \Delta(E) = (\ch{1}(E))^2 - \ch{2}(E)\ch{0}(E)$.  The
discriminant is non-negative on $\mu_H$-semistable sheaves by
\eqref{eq:bg-inequality}. We seek a form, analog to $\Delta$, which involves the
whole $\ch{orb}$ and enjoys a similar positivity property. We use the notation
of Section \ref{sec_McKay_corr} throughout.

Let $E\in\Coh(\sS)$, define its \textit{orbifold discriminant} $\Delta_{orb}( E
)\coloneqq \Delta(\Phi( E ))$. The goal of this section is to prove the
following:

\begin{thm}\label{thm_discriminant}
  Let $H$ be ample on $\sS$. Let $ E $ be a $\mu_H$-semistable sheaf on $\sS$.
  Let $\tilde E=\Phi  E $ be its image on $S$. Then, $\Delta_{orb}( E
  )=\Delta(\tilde E)\geq 0$.
\end{thm}

First, observe the following lemma:

\begin{lemma}
  Let $E$ be a torsion-free sheaf on $\sS$, then $\tilde{E}$ is a sheaf.
\end{lemma}

\begin{proof}
If $\tilde{E}$ is not a sheaf, then there is an exact sequence  
\[ 
  H^{-1}(\tilde E )[1]\to \tilde E  \to H^0(\tilde E )
\]
in $\zPer(\tilde S/S)$, where $H^{-1}(\tilde E )$ is a torsion sheaf, since it
lies in $\tilde{\sF}_0$. Applying $\Phi^{-1}$ to the sequence above one obtains
a short exact sequence in $\Coh(\sS)$
\[ 
  B \to  E  \to A. 
\]
with $B\in \sF_0$. We have $\pi_*(B)=f_*(H^{-1}(\tilde{E}))=0$ by definition of $\tilde{\sF}_0$, so $B$ is a torsion sheaf and $ E $ is not torsion-free. 
\end{proof}

\begin{definition}
We say that a sheaf $E$ on $\sS$ (resp. $\tilde{E}$ on $\tilde S$)
\textit{descends to $S$} if the natural map 
\[ 
  \pi^*\pi_*E \to E 
\]
(resp. $f^*f_*E \to E$) is an isomorphism.
\end{definition}

\begin{lemma}\label{lem_torfree_iff_torfree+descends}
  $\tilde{E}$ is a torsion-free sheaf if and only if $ E $ is torsion-free and it descends to $S$. 
\end{lemma}

\begin{proof}
By the previous lemma, if $ E $ is torsion-free then $\tilde{E}$ is a sheaf. Now
we show that if, additionally, $ E $ descends to $S$, then $\tilde{E}$ is
torsion-free. Suppose $\tilde{E}$ has a torsion subsheaf $\tilde F$, for sake of
contradiction. Then, applying $\Phi^{-1}$ to the sequence $\tilde F\to\tilde
E\to\tilde E'$ and taking the associated long exact sequence of sheaves, one
gets 
\[ 
  H^0(\Phi^{-1}\tilde F) \to  E  \to  E' \to H^1(\Phi^{-1}\tilde F) 
\]
(the other terms in the long exact sequence of cohomology sheaves vanish because
$E$ is a sheaf, and $H^{-1}( E ')=0$ since images of sheaves under $\Phi^{-1}$
may only have cohomologies in degrees 0,1).
The sheaf $H^0(\Phi^{-1}\tilde{F})$ is torsion: this follows from the triangle 
\[ H^0(\Phi^{-1}\tilde{F}) \to \Phi^{-1}\tilde{F} \to H^1(\Phi^{-1}\tilde{F})[-1] \]
and the fact that $H^1(\Phi^{-1}\tilde{F})\in \sF_0$ is torsion, and $\Phi^{-1}\tilde{F}$ has rank 0 since $\Phi$ and $\Phi^{-1}$ preserve ranks. Therefore, either $ E $ has torsion,
or fits in a short exact sequence 
\begin{equation}\label{eq_saturated}
     E  \to  E ' \to F
\end{equation}
where $F\coloneqq \Phi^{-1}\tilde F\in \sF_0[-1]$.
Then, there is a diagram with exact rows and columns
\begin{equation*}
\begin{tikzcd}[row sep=small]
\pi^*\pi_* E  \rar \dar &   E  \rar \dar & 0 \dar \\
\pi^*\pi_* E ' \arrow{r} \dar  &  E ' \arrow[d]\arrow{r}  & K \dar\\ 
M  \arrow{r}  & F\arrow{r}  & K 
\end{tikzcd}
\end{equation*}
where the top right corner is zero since $ E $ descends to $S$, and $M$ is a
repeated extension of pull-backs of $\sO_p$ from $S$. The middle row of the
diagram shows $\dR\pi_*K=0$. If $M\neq 0$, we have $\pi_*M\neq 0$, contradicting
$\pi_*F=f_*\tilde{F}=0$. If $\sM=0$, then $F\simeq K$ but $K\in \tilde{\sT}_0$
while $F\in\tilde{\sF}_0[-1]$.

Conversely, if $ E $ has torsion then $\tilde{E}$ is either not a sheaf, or it
has torsion, so we may assume that $ E $ is torsion-free. Then $ E $ fits in a
sequence \eqref{eq_saturated}, where $F[-1]$ is the image of the torsion of
$\tilde{E}$. If $ E $ does not descend to $S$, we have $K\neq 0$ in the diagram 
\begin{equation*}
\begin{tikzcd}[row sep=small]
\pi^*\pi_* E  \rar \dar &   E  \rar \dar & K \dar{a} \\
\pi^*\pi_* E ' \arrow{r} \dar  &  E ' \arrow[d]\arrow{r}  & K' \\ 
M  \arrow{r}  & F & 
\end{tikzcd}
\end{equation*}
If $F=0$ we must have $\ker(a)\simeq M$, which is a contradiction because $\dR\pi_*K=\dR\pi_*K'=0$ while $\dR\pi_*M\neq 0$. We showed that $\tilde{E}$ has torsion whenever $ E $ has torsion or does not descend.
\end{proof}

Next, we show:

\begin{lemma}\label{lem_sat+semistable__semistable}
  Suppose $ E $ descends to $S$ and is $\mu_H$-semistable, then $\tilde{E}$ is
  $\mu_H$-semistable of the same slope.
\end{lemma}

\begin{proof}
  Since $ E $ descends and is torsion-free, $\tilde{E}$ is torsion-free by Lemma
  \ref{lem_torfree_iff_torfree+descends}.  Suppose $\tilde{E}$ is destabilized
  by a sequence $\tilde F\to \tilde E\to \tilde K$ with $\tilde F$ torsion-free.
  Apply $\Phi^{-1}$ to the sequence and get a triangle
  \[ 
    F \to  E  \to K 
  \]
  where $ E $ is a sheaf, $F$ is a sheaf since $\tilde F$ is torsion-free, and
  $K$ has cohomologies in degree 0 and -1. This implies immediately that $K$ is
  just a sheaf, and $F\to  E $ is still injective. However,
  $\mu_H(F)=\mu_H(\tilde F)$ and $\mu_H( E )=\mu_H(\tilde E)$, so $ E $ is also
  unstable. 
\end{proof}

%\begin{remark}
%Probably a converse holds as well, but I don't think we need it.
%\end{remark}

\begin{corollary}\label{cor_descend+sst_gives_thm}
  Let $E\in\Coh(\sS)$, and $\tilde E\coloneqq \Phi(E)$. If $ E $ descends to $S$
  and is $\mu_H$-semistable, then $\Delta(\tilde E)\geq 0$.
\end{corollary}

\begin{proof}
  By Lemma \ref{lem_sat+semistable__semistable}, $\tilde E$ is slope semistable,
  so it satisfies the Bogomolov inequality on $\tilde S$.
\end{proof}

We are finally ready to prove Theorem \ref{thm_discriminant}:

\begin{proof}[Proof of \ref{thm_discriminant}]
  If $E$ descends to $S$, then $\tilde E$ is slope-semistable, and Cor.
  \ref{cor_descend+sst_gives_thm} applies. Otherwise, the sheaf $\tilde{E}$ may
  have a torsion subsheaf $\tilde T$, supported on the exceptional locus, and
  fit in a sequence 
  \[ 
    \tilde T\to\tilde E \to\tilde E' 
  \]
  where $\tilde E'$ is torsion-free. In particular, $\mu_H(\tilde
  E)=\mu_H(\tilde E')$. Notice moreover that $\tilde T\in \tilde\sF_0$,
  otherwise $E$ would have torsion.  Now consider 
  \begin{equation}\label{eq_expression_for_Delta}
    \Delta(\tilde E)= \Delta(\tilde E') + 2\ch 1(\tilde E')\ch 1(\tilde T)+  \ch 1(\tilde T^2) - 2\ch 0(\tilde E')\ch 2(\tilde T).  
  \end{equation}
  First, we claim $\tilde E'$ is $\mu_H$-semistable. We show this by showing
  that $ E '$ is semistable, and applying Lemma
  \ref{lem_sat+semistable__semistable}. A destabilizer $F'$ of $ E '$ either
  factors through the inclusion $ E  \to  E '$, or it fits in a diagram 
  \begin{equation*}
  \begin{tikzcd}[row sep=small]
  F \arrow{r} \dar  & F' \arrow[d]\arrow{r}  & T' \dar\\ 
  E   \arrow{r}  &  E '\arrow{r}  & T 
  \end{tikzcd}
  \end{equation*}
  where $T=\Phi^{-1}\tilde T[1]$ is a sheaf, $T'$ is the image of $F'$ in $T$,
  and $F$ is a subsheaf of $ E $ with $\mu_H(F)=\mu_H(F')>\mu_H( E ')=\mu_H( E
  )$, which destabilizes $ E $. Hence we have $\Delta(E')\geq 0$.

  Observe that the summand $2\ch 1(\tilde E')\ch 1(\tilde T)$ in
  \eqref{eq_expression_for_Delta} vanishes. In fact, if $ \tilde E '$ descends
  to $S$, then $\ch 1(\tilde E')$ is orthogonal to the exceptional curve, and
  hence it's orthogonal to $\ch 1(\tilde T)$. 

  It remains to understand the last two summands. By construction, the sheaf $T$
  is a repeated extension of at most $\rk(E')$ proper quotients of clusters (see
  Sec. \ref{sec_McKay_corr}). Lemma \ref{lem_inequality_on_LDs} below, applied
  with $M=\rk(E')$, shows that
  \begin{equation}
    \ch 1(\tilde T^2) - 2\ch 0(\tilde E')\ch 2(\tilde T)\geq 0.
  \end{equation}
  The conclusion is that $\Delta(\tilde E)\geq 0$. 
\end{proof}

To prove Lemma \ref{lem_inequality_on_LDs}, we need to understand the structure
of proper quotients of clusters:

\begin{lemma}\label{lem_structureQuotientCluster}
  Quotients of clusters correspond under the McKay functor $\Phi$ to sheaves
  $L_D\subseteq \omega_C$ whose Chern character satisfies $\ch{}(L_D)=(0,D,-1)$,
  where $D\leq C$ is an effective one dimensional cycle whose coefficients are
  those of a root in the root system associated to the singularity. In
  particular, if $D_1$ and $D_2$ are two such cycles, then $D_1.D_2\geq -2$.
\end{lemma}

\begin{proof}
  Consider the proper quotient $Q$ of a cluster $M$, and the diagram with exact rows
  \[
    \begin{tikzcd}[row sep=small]
    K'\rar \dar & M\rar \dar & Q \dar \\
    K \rar & M \rar & \sO_p
    \end{tikzcd}
  \]

  Denote by $H$, resp. $H'$, the images of $K$, resp. $K'$ under the McKay
  functor. Since $H'$ is built of repeated extensions of $\sO_{C_i}(-1)$, we get
  immediately that $L\coloneqq \Phi(Q)[-1]$ satisfies $[L]=[H'] - [\C_p]$ and
  has Chern character $\ch{}(L)=(0,D,-1)$ where $D$ is a positive linear
  combination of the $C_i$. Moreover, the image under $\Phi$ of the diagram
  above exhibits $L$ as a subobject of $\Phi(\sO_p)[-1]=\omega_C$. 

  What is left to argue is that subobjects of $\omega_C$ with Chern character
  $(0,D,-1)$ must satisfy that the coefficients of $D$ are those of a positive
  root of the associated root system.  Suppose $L$ is as above. Since
  $\iota_*L\subseteq \iota_*\omega_C$, we must have
  \[
    \ch 2 \iota_*L \leq \ch 2 \iota_*(\omega_C)_{|D} - (C-D).D
  \]
  (its degree cannot exceed that of the restriction of $\iota_*\omega_C$, but
  its sections must additionally vanish along the intersection between $D$ and
  its complement). Now $\ch 2 \iota_*L=-1$, and $\ch 2 \iota_*(\omega_C)_{|D}= C.D
  - D^2/2$ by the Grothendieck--Riemann--Roch theorem. Then the inequality above
  reads
  \[ 
    -2\leq D^2.
  \]
  Identify $D$ with an element $\alpha$ of the root lattice. Then the inequality states 
  \[
    2\geq \pair{\alpha,\alpha}
  \]
  which can only happen (realizing equality) if $\alpha$ is a root. For the last
  statement, it suffices to observe that if $D_1$ and $D_2$ correspond to roots
  $\alpha$ and $\beta$, then $D_1.D_2=-\pair{\alpha,\beta}\geq -2$. This follows
  from the theory of simply laced root systems.
\end{proof}

We can then show:

\begin{lemma}\label{lem_inequality_on_LDs}
  Let $T$ be a repeated extension of at most $M$ proper quotients of clusters.
  Denote $\tilde T\coloneqq \Phi(T)$, then 
  \begin{equation}\label{eq_Delta(T,E)}
    \ch 1(\tilde T^2) - 2M\ch 2(\tilde T)\geq 0.
  \end{equation}
\end{lemma}

\begin{proof}
  Quotients of clusters correspond to objects $L_D$ as described in Lemma
  \ref{lem_structureQuotientCluster}.

  Since $[\tilde T]=\sum\limits_{j=1}^m[L_{D_j}]$ in $K(S)$, with $m\leq M$, we
  have $\ch{}(\tilde T)=(0,\sum_j D_j, -m)$. Then
  \begin{align*}
    \ch 1(\tilde T^2) - 2M\ch 2(\tilde T) =  (\sum\limits_{i=1}^{m}D_i)^2 + 2Mm\geq 
    -2m^2 + 2Mm \geq 0,
  \end{align*}
  where the first inequality is a consequence of the last statement of Lemma
  \ref{lem_structureQuotientCluster}. 
\end{proof}

\section{Stability conditions on $D^b(\sS)$}

In this section we construct stability conditions on $D^b(\sS)$.

\subsection{Background on Bridgeland stability}

 We recall some aspects of the theory of stability conditions in what follows, and direct the interested reader to the
seminal works of Bridgeland \cite{Bri07_triang_cat}, \cite{Bri08_k3} and to the
survey \cite{MS17}. Let $\sD$ be a triangulated category.

\begin{definition}
  A \emph{heart of a bounded t-structure} in $\sD$ is a full additive
  subcategory $\sA\subset \sD$ satisfying the following properties:
  \begin{enumerate}[(i)]
    \item $\Hom^i(A,B)=0$ for $i<0$;
    \item every object in $\sD$ has a filtration by cohomology objects in
      $\sA$. In other words, for all non-zero $E\in \sD$  there are integers
      $k_1>...>k_m$ and a collection of triangles 
  \begin{equation*}
  \begin{tikzcd}[column sep=small]
    0=E_0 \arrow{rr}  & & E_1 \arrow[dl]\arrow{rr}  & & E_2 \arrow[dl] \arrow{rr}  & & ...\arrow{rr}  & & E_{m-1}\arrow{rr}  & & E_m=E\arrow[dl] \\
    & A_1 \arrow[ul, dashed] & & A_2 \arrow[ul, dashed] & & & &  & & A_m \arrow[ul, dashed] &
  \end{tikzcd}
  \end{equation*}
  where $A_i[-k_i]\in\sA$.
  \end{enumerate}
\end{definition}

One can check that if $\sA$ is the heart of a bounded t-structure in $\sD$, then $\sA$ is abelian.% \AAA{ref}.

The definition of a stability condition also involves the choice of a finite rank lattice $\Lambda$ and a surjective group homomorphism $v\colon K(\sD) \twoheadrightarrow \Lambda$.

\begin{definition}\label{def_pre_stab_cond}
  A \emph{pre-stability condition} on $\sD$ is a pair $\sigma=(Z,\sA)$ where:
  \begin{enumerate}[(i)]
    \item $\sA$ is the heart of a bounded t-structure in $D^b(\sS)$;
    \item $Z\colon \Lambda \to \C$ is an additive homomorphism called the \emph{central charge};
  \end{enumerate}
  and they satisfy the following properties:
  \begin{enumerate}
    \item For any non-zero $E\in \sA$, $$ Z(v(E))\in \R_{>0} \cdot
      e^{i\pi\phi}$$ with $\phi\in(0,1]$. Define the \emph{phase} of $0\neq
      E\in\sA$ to be $\phi(E)\coloneqq \phi$. We say that $E\in\sA$ is
      $\sigma$\emph{-semistable} if for all non-zero subobjects $F\in\sA$ of
      $E$, $\phi(F)\leq \phi(E)$. $E$ is $\sigma$\emph{-stable} if for all
      non-zero proper subobjects $F\in\sA$ of $E$, $\phi(F)< \phi(E)$. We denote by
      $\PP(\phi)$ the category of semistable objects of phase $\phi$.
    \item (HN filtrations) The objects of $\sA$ have \emph{Harder-Narasimhan
      filtrations} with respect to $Z$. In other words, for every $E\in\sA$
      there is a unique filtration $$ 0=E_0\subset E_1\subset ... \subset
      E_{n-1}\subset E_n=E $$ such that the quotients $E_i/E_{i-1}\in
      \PP(\phi_i)$ with $\phi_1>\phi_2>...>\phi_n$.
  \end{enumerate}
\end{definition}

\begin{definition}\label{def_supp prop}
  A pre-stability condition $\sigma$ is a \textit{stability condition} if it
  additionally satisfies the \textit{support property}, i.e. 
  \[
    C_\sigma\coloneqq \inf\left\lbrace
    \dfrac{\abs{Z(v(E))}}{\norm{v(E)}} \,\colon\, 0\neq
    E\in\PP(\phi),\,\phi\in\R \right\rbrace >0.
  \]
\end{definition}

There is an alternative characterization of the support property, given by the
following Proposition \cite[Section 2.1]{KS08}:

\begin{proposition}\label{prop_supp prop quad form}
  A pre-stability condition $\sigma=(Z,\sA)$ satisfies the support property if
  and only if there exists a quadratic form $Q$ such that $Q$ is negative
  definite on the kernel of $Z$, and $Q(E)\geq 0$ for every $\sigma$-semistable
  object $E$ in $\sA$.
\end{proposition}

Let $\Stab(\sD)$ denote the set of stability conditions on $\sD$. In \cite[Sec.
6]{Bri07_triang_cat}, the author defines a generalized metric $f$ on $\Stab(\sD)$ which makes it into a topological space. Moreover, a deformation result holds:

\begin{thm}[{\cite[Thm. 7.1]{Bri07_triang_cat}}]\label{thm_Deformation_Bridgeland}
  Let $\sigma=(Z,\sA)\in \Stab(\sD)$. Then for every $\epsilon>0$ there exists a
  disc $\Delta\subset \Hom(\Lambda,\C))$, centered at $Z$, such that for every
  $W\in \Delta$ there exists a stability condition $\tau=(W,\sA')$ with
  $f(\sigma,\tau)<\epsilon$.
\end{thm}

In turn, this leads to the following:

\begin{thm}[{\cite[Thm. 1.2]{Bri07_triang_cat}}]\label{thm_BriLocalHomeom}
  The central charge map $\varpi\colon\Stab(\sD)\to \Hom(\Lambda,\C)$ given by
  $(Z,\sA)\mapsto Z$ is a local homeomorphism. In particular, $\Stab(\sD)$ is a
  complex manifold of dimension $\rk(\Lambda)$.
\end{thm}

In what follows, we set $\sD = D(\sS)$, we choose $v$ to be the the orbifold Chern character $\ch{orb}$ (Definition \ref{def_orbifold_Chern_char}), and let $\Lambda$ denote its image in $H^*_{orb}(\sS)\coloneqq H^*(I\sS)$.

\begin{comment}
shows that the function
\begin{equation}\label{eq_SlicingMetric}
  f(\sigma,\tau)=\sup_{0\neq E\in
  \sD}\set{\abs{\phi^+_\sigma(E)-\phi^+_\tau(E)},\abs{\phi^-_\sigma(E)-\phi^-_\tau(E)}}
\end{equation}
is a generalized metric on $\Stab(\sD)$ which makes it into a topological space.
\end{comment}

%A part of this work will be dedicated to the study of the map $\pi$. This will require the following lemma.

%\begin{lemma}[{\cite[Lemma 6.4]{Bri07_triang_cat}}]\label{lem_BriCloseAreEqual}
%Let $\sigma$, $\tau\in\Stab(\sD)$ be stability conditions with $\pi(\sigma)=\pi(\tau)$. If $f(\sigma,\tau)<1$, then $\sigma=\tau$.
%\end{lemma}

\subsection{Construction of a pre-stability condition}\label{sec_stab_fun}

Again, \(\sS\) denotes a projective Kleinian orbisurface with a unique
stacky point \(p\in \sS\) with residual gerbe \(BG=[\ast/G]\) and ample divisor
\(H\). Set \(\iota\colon BG\hookrightarrow \sS\) the corresponding closed
substack. 

Define subcategories of \(\Coh(\sS)\) by
\begin{align*}
  T_{H,b} &= \{E\in \Coh(\sS)\mid \mbox{ for all $\mu_H$-semistable factors $F$ of $E$, } \mu_{H}(F)> b\}; \\
  F_{H,b} &= \{E\in \Coh(\sS)\mid \mbox{ for all $\mu_H$-semistable factors $F$ of $E$, } \mu_{H}(F)\leq b\}.
\end{align*}

The existence of Harder-Narasimhan filtrations for slope stability in our
context is proven exactly as in the case of schemes, which is detailed in
\cite[Sec. 1.6]{HL10}. As a consequence,$\left(T_{H,b},  F_{H,b}\right)$ is a
torsion pair, so we can perform the usual tilt and obtain the heart of a bounded
t-structure
\[
  \mathcal \Coh^{b}(\sS)\coloneqq \left( F_{H,b}[1],  T_{H,b} \right).
\]

Since $\delta(F)$ is a linear function of $\ch{orb}(F)$ (see Section
\ref{sec_Chern_classes}), the function $Z_{w,\gamma}\colon \Lambda \to \C$
defined as
\[ 
  Z_{w,\gamma}(E)=Z(\ch{orb}(E))= -\ch 2(E) + w \ch 0(E) + \gamma.\delta(E) +iH.\ch 1 (E)
\]
is also linear (here, $w\in\C$ and $\gamma\in \R$). Note that we identify $\rk$ and $\ch 0$ here, slightly deviating from the usual notation of having $H^2\ch 0$ as a summand in the central charge. The goal of this section is
to prove the following theorem. We denote $N\coloneqq \abs{G}$ and $D\coloneqq
\delta(\sO_{\sS})$.

\begin{thm}\label{thm_OurStabilityCondition}
  Choose parameters $\gamma\in (0,\frac{1}{N-1})$ and $w\in \C$ such that:
  \begin{enumerate}[(i)]
    \item $\re w > - \frac{(\im w)^2}{H^2}+ (2+\gamma)D -  (1+\gamma)^2$;
    \item $\re w >\frac12 \frac{(\im w)^2}{H^2}  - \gamma(D-\frac{N-1}{N})>0$.
  \end{enumerate} 
  Then, the pair $(Z_{w,\gamma}, \Coh^{-\im w}(\sS))$ is a stability condition
  on $D(\sS)$.
\end{thm}

We split the proof of the theorem in two sections: the present Section
\ref{sec_stab_fun} contains the construction of a pre-stability condition, while
Section \ref{sec_supp_prop} contains arguments about the support property and
concludes the proof. First, we prove a preliminary lemma: 

\begin{lemma}\label{lem_estimate_delta_for_torfree}
  Let $F$ be a torsion free sheaf on $\sS$. Consider the sequence
  \[ 
    E\coloneqq\pi^*\pi_* F \to F \to M 
  \]
  Then $\delta(F)=\delta(E)+\delta(M)\geq (\rk F)(D - \frac{N-1}{N})$.
\end{lemma}

\begin{proof}
  The sheaf $M$ is torsion, supported on the stacky point, and obtained by
  repeated extensions of copies of $\sO_p\otimes\rho_i$ (because it pushes forward
  to zero). Every $\sO_p\otimes\rho_i$ appears in the composition series of $M$ at
  most $\rk(F)\cdot \dim \rho_i$ times. Every one of
  these copies contributes $\delta(\sO_p\otimes\rho_i)=-\dim(\rho_i)/N$, so 
  \[
    \delta(M)\geq (\rk(F)) \sum (r_i) \delta(\sO_p\otimes\rho_i)\geq
    -(\rk(F))\frac{N-1}{N}.
  \]

  Now, $[E]$ is in the span of $[\sO]$, $[\sO_q]$ (for $q\neq p$) and $NS(S)$, write
  \[ 
    [E]=(\rk F)[\sO] + a [\sO_q] + \phi 
  \]
  where $\phi\in NS(S)$ and observe that $\delta(\sO_q)=0$ and $\delta(\phi)=0$: this can be checked by observing
  that that $[\dL\iota^*\sO_K]=0$, where $[K]=\phi$ is the class of any curve on
  $\sS$. So $\delta(E)=\rk(F)D$.
\end{proof}

\begin{lemma}\label{lem_stab_function}
  Provided that $0< \gamma <\frac{1}{N-1}$  and $\re w -\frac12 \frac{(\im
  w)^2}{H^2}  + \gamma(D-\frac{N-1}{N})>0$, the group homomorphism
  $Z_{w,\gamma}$ is a stability function on $\sA\coloneqq \Coh^{-\im w}(\sS)$.
\end{lemma}

\begin{proof}
  First, one observes that for a sheaf $E$ we have
  \[
    \frac{H.\ch 1(E)}{\ch 0 (E)} = \mu(E)>-\im w 
  \]
  if and only if $E\in T_{H,-\im w}$, so that $\im Z(E)\geq 0$ for all $E\in
  \sA$.

  Then, we only need to check that $\re Z(E)<0$ whenever $\im Z(E)=0$. Assume
  then that $\im Z(E)=0$. If $E$ is torsion, then $\ch 0 (E)=0$ and by $\im
  Z(E)=0$ we get $\ch 1(E)=0$. So $E$ is supported on points. 

  We have that $\re Z(E)<0$ for all $E$ supported on points: $\re Z(\sO_p) =
  -1/N + \gamma(1-1/N)$ and $\re Z(\sO_p\otimes\rho)=-\dim(\rho)/N +
  \gamma(-\dim(\rho)/N)$ (see Lemma \ref{lem_delta_of_skyscrapers}) are both
  negative by the assumptions on $\gamma$, and every $E$ is an extension of
  these.

  The other case to check is the following: $\re Z(E)>0$ (because $E[1]\in \sA$)
  for all $\mu$-semistable sheaves $E$ with $\mu(E)=-\im w$. In this case, the
  Bogomolov-Gieseker inequality \eqref{eq:bg-inequality} yields $-\ch 2 \geq
  -\frac{\ch 1^2}{2\ch 0(E)}$, since $\ch 0(E)>0$. Hence 
  \begin{align*}
    \re Z(E)= -\ch 2(E) + \re w \ch 0(E) + \gamma.\delta(E) \\
    \geq \ch 0(E)\left[ \re w - \frac{\ch 1(E)^2}{2\ch 0(E)^2}\right]+ \gamma.\delta(E)
  \end{align*}

  Now, since $\mu(E)=-\im w$ we have 
  \[
    \left( \ch 1(E)+\frac{\im w\ch 0(E)}{H^2}H \right)^2\leq 0
  \]
  by the Hodge index theorem. Expanding this and using $\mu(E)=-\im w$
  once more, one gets
  \begin{equation}
    \label{eq_inequality on ch1_square}
    -\ch 1(E)^2 \geq - \frac{(\im w)^2(\ch 0(E))^2}{H^2}
  \end{equation}

  Combining the inequalitites above, and using Lemma  \ref{lem_estimate_delta_for_torfree} to estimate $\delta(E)$, we get:
  \begin{align*}
    \re Z(E)\geq \ch 0(E)\left[ \re w - \frac{\ch 1(E)^2}{2\ch 0(E)^2}\right]+ \gamma \delta(E)\\
   \geq \ch 0(E) \left[ \re w -\frac12 \frac{(\im w)^2}{H^2} \right] +  \gamma\delta(E)\\
   \geq  \ch 0(E) \left[ \re w -\frac12 \frac{(\im w)^2}{H^2}  + \gamma(D-\frac{N-1}{N})\right].
  \end{align*} 
  This last part is positive by the assumption on $\re w$.
\end{proof}

\begin{lemma}\label{lem_HN_filtrations}
  If $H$ is a rational class and $\im w\in\Q$, then $Z$ satisfies Harder-Narasimhan filtrations on $\sA$.
\end{lemma}

\begin{proof}
  Under the rationality assumptions, it is easy to see that the image of $\im
  Z_{w,\gamma}$ is discrete. Then, it is enough to show that $\sA$ is Noetherian
  \cite[Prop. 4.10]{MS17}. This is proven exactly in the same way as \cite[Lemma
  6.17]{MS17}.
\end{proof}

Summarizing this discussion:

\begin{proposition}\label{prop_pre_stability_rational}
  Let $H$ be a rational class, and $\gamma$ and $w$ chosen so that $0< \gamma
  <\frac{1}{N-1}$, $\re w -\frac12 \frac{(\im w)^2}{H^2}  +
  \gamma(D-\frac{N-1}{N})>0$ and $\im w\in\Q$.  Then, the pair
  $\sigma_{w,\gamma}=(Z_{w,\gamma}, \sA)$ is a pre-stability condition on
  $D(\sS)$.
\end{proposition}

\begin{proof}
  The previous Lemmata \ref{lem_stab_function} and \ref{lem_HN_filtrations} show
  that $(Z_{w,\gamma}, \sA)$ satisfies Definition \ref{def_pre_stab_cond}.
\end{proof}

\subsection{Support property}\label{sec_supp_prop}

We now define a quadratic form $Q$ on $K(\sS)$ which is negative definite on
$\ker Z_{w,\gamma}$ and $Q(E)\geq 0$ for all $E\in \sA$ which are
$\sigma_{w,\gamma}$-semistable. To do so, we need to investigate objects which
are semistable for a limiting value of $\re w$:

\begin{lemma}\label{lem_classification_stable_sheaves}
  Let $\mathcal{R}$ be the set of objects $E\in \sA$ that are
  $\sigma_{w,\gamma}$-semistable for all $\alpha\coloneqq \re w \gg 0$. Then
  every $E\in \mathcal{R}$ has one of the following forms:
  \begin{enumerate}
    \item $E$ is a slope semistable sheaf.
    \item $H^{-1}(E)=0$ and $H^0(E)$ is torsion;
    \item $H^{-1}(E)$ is a torsion free, slope semistable sheaf, and $H^0(E)$ is supported on points.
  \end{enumerate}
\end{lemma}

\begin{proof}
  The proof of \cite[Lemma 6.18]{MS17} carries over \textit{mutatis mutandis}.
\end{proof}

\begin{definition}
  Define a preliminary quadratic form 
  \[
    Q_{0}(E)\coloneqq \Delta_{orb}(E). 
  \]
\end{definition}

\begin{comment}
\begin{remark}
The quadratic form is what we called $\Delta_{st}$ before. In other words, if $\sE=\Phi(E)$ under the McKay correspondence, we have 
\[ Q_0(\sE)=\Delta_{st}(\sE)=\Delta(E). \]
In other words, this matches Tramel's preliminary quadratic form, just translated in stacky terms. The ugliest aspect is that the rank is contributing to $\delta$, so it shows up where we would not want it.
\end{remark}
\end{comment}

The form $Q_0$ is the pull-back on the cohomology of the stack of $\Delta=\ch
1^2-2\ch 0\ch 2 $ on the surface under the McKay correspondence. The image in
$N(\sS)$ of a numerical class $\bfw\coloneqq (r,\phi + \sum t_jC_j, d)$ on $\tilde{S}$
is
\begin{equation}
    \label{eq_class_v}
    \bfv \coloneqq r[\sO] + \phi + d[\sO_{Np}] + \sum_i (t_i+d r_i)[\sO_p\otimes\rho_i]
\end{equation}

Applying Lemma \ref{lem_delta_of_skyscrapers} to \eqref{eq_class_v} yields:
\begin{align}\label{eq_compute_ch2(v)}
    \ch 2(\bfv)= d+\frac 1N\sum r_i t_i;\\
    \delta(\bfv)= rD - \frac 1N\sum r_it_i;\label{eq_compute_delta(v)}
\end{align}

\begin{lemma}\label{lem_Q0_neg_def}
 The form $Q_0$ is negative definite on $\ker Z_{w,\gamma}$, as long as $\re w >
 (2+\gamma)D -  (1+\gamma)^2 - \frac{(\im w)^2}{H^2}$.
\end{lemma}

\begin{proof}
  Keep the notation as above, and suppose $\bfv$ belongs to $\ker Z$. The
  condition on the real part reads $\ch 2 (\bfv) = \re wr+\gamma\delta(\bfv)$.
  One sees from \eqref{eq_compute_ch2(v)} and \eqref{eq_compute_delta(v)} that
  $d=\ch 2 (\bfv) + \delta(\bfv) - rD$. Then, $Q_0(\bfv)$ rewrites as 
  \begin{align*}
    Q_0(\bfv)  =\Delta(\bfw) &= \phi^2 + \left(\sum t_jC_j\right)^2 -2r(d)\\
    &= \phi^2 + \left(\sum t_jC_j\right)^2 -2r(\ch 2 (\bfv) + \delta(\bfv) - rD)\\
    &= \phi^2 + \left(\sum t_jC_j\right)^2 -2r((\re w-D)r + (1+\gamma)\delta(\bfv)).
  \end{align*}
  Now, \eqref{eq_compute_delta(v)} gives 
  \begin{equation}
    \begin{split}\label{eq_expand_Q0}
    Q_0(\bfv)&= \phi^2 -2r^2(\re w-D) + \left(\sum t_jC_j\right)^2 -2r (1+\gamma)\left(rD - \frac1N \sum r_it_i\right)\\
    &= \phi^2 -2r^2(\re w-D - (1+\gamma)D) + \left(\sum t_jC_j\right)^2 +2r (1+\gamma)\left(\frac1N \sum r_i t_i\right).    
    \end{split}
  \end{equation}
  We concentrate now on the quantity
  \begin{equation}
    \label{eq_control_Ci^2}
    \left(\sum t_jC_j\right)^2 +2r (1+\gamma)\left(\frac{\sum r_it_i}{N}\right). 
  \end{equation}
  By adding and subtracting $\left(\frac{\sum r_it_i}{N}\right)^2$ and completing a square, we have 
  \begin{equation}\label{eq_complete_square}
    \begin{split}
    \left(\sum t_jC_j\right)^2 +2r (1+\gamma)\left(\frac{\sum r_it_i}{N}\right) = \\    
    \left[\left(\sum t_iC_i\right)^2 + \left(\frac{\sum r_it_i}{N}\right)^2\right] - \left[\frac{\sum r_it_i}{N} - r(1+\gamma)\right]^2 + r^2(1+\gamma)^2.
    \end{split}
  \end{equation}
  The condition on the imaginary part is $\im Z(v)=H.\phi+ \im w r=0$. Hence the
  Hodge index theorem yields $\phi^2 \leq  \frac{(\im w)^2r^2}{H^2}$ as in
  \eqref{eq_inequality on ch1_square}. Combine this with equations
  \eqref{eq_expand_Q0} and \eqref{eq_complete_square} to obtain
  \begin{equation}
    Q_0(\bfv)\leq -2r^2\left[\re w - (2+\gamma)D +  (1+\gamma)^2 + \frac{(\im w)^2}{H^2} \right] 
    + \left[\left(\sum t_iC_i\right)^2 + \left(\frac{\sum r_it_i}{N}\right)^2\right].
  \end{equation}
  The second summand is negative unless $t_i=0$ for all $i$ by Lemma
  \ref{lem_Toeplitz} below. In this case, the first summand is negative unless
  $r=0$. If $r=t_i=0$, then we must have $\delta(\bfv)=0$, $\ch 2(\bfv)=0$, and
  $H.\phi=0$, which implies $Q_0(\bfv)=\phi^2$ is negative definite by the Hodge
  index theorem. This concludes the proof.
\end{proof}

\begin{lemma}\label{lem_Toeplitz}
  The quantity $\left[\left(\sum t_iC_i\right)^2 + \left(\frac 1N\sum
  r_it_i\right)^2\right]$ is non-positive, and it is zero only if all $t_i=0$.
\end{lemma}

\begin{proof} 
  Let $H$ denote the intersection matrix of the exceptional curves. Its negative
  $-H$ is the Cartan matrix associated with the root system corresponding to the
  singularity. Let $J$ denote the matrix associated with
  the symmetric bilinear
  form $(t_1,...,t_M)\mapsto \left(\sum r_it_i\right)^2$. It is sufficient to
  prove that the matrix 
  \[
    A\coloneqq H + \frac{1}{N^2}J 
  \]
  is negative definite. To do so, we study the eigenvalues of $H$ and $J$. The
  entries of $J$ are 
  \[ 
    (J)_{i,j}=r_ir_j, 
  \]
  and $J$ can be written as $\mathbf{r} \mathbf{r}^T$ where
  $\mathbf{r}=(r_1,...,r_M)$. Then, $J$ has rank 1 and an eigenvector is
  $\mathbf{r}$ with eigenvalue $\sum_{i=1}^M r_i\leq N-1$. All other eigenvalues
  are 0. 

\begin{remark}
  From the representation-theoretic viewpoint, the quantity 
  \[
    h\coloneqq\sum_{i=0}^M r_i=\sum_{i=1}^M r_i + 1
  \] 
  coincides with the Coxeter number of the root system associated with the
  singularity, since the $r_i$ are the coefficients of the longest root.
\end{remark}

  Write $\alpha_1\geq ... \geq \alpha_{N-1}$ for the eigenvalues of $A$. By
  Weyl's inequality on sums of symmetric matrices, we have 
  \[ 
    \alpha_1\leq \eta_1 + \frac{h-1}{N^2}, 
  \]
  where $\eta_1$ is the biggest eigenvalue of $H$. The eigenvalues of $H$ are
  computed in \cite{Dam11}, and we have that
  \begin{equation}
    \eta_1= -2 + 2\cos \left(\frac{\pi}{h}\right). 
  \end{equation}
  The Coxeter numbers and orders of the groups are

  \begin{center}
  \begin{tabular}{c|c|c}
    Root system & $h$ & $N$  \\ \hline
    $A_{n-1}$ & $n$ & $n$ \\
    $D_n$ & $2n$ & $4n$  \\
    $E_6$ & 12 & 24  \\
    $E_7$ & 18 & 48  \\
    $E_8$ & 30 & 120 
  \end{tabular}
  \end{center}
  It is then straightforward to check that $\eta_1 + \frac{h-1}{N^2}<0$ in all
  the cases listed, and $A$ is therefore negative definite.
\end{proof}

\begin{comment}
\begin{equation}
    \begin{cases}
    \eta_i= -2 + 2\cos \left(\frac{i\pi}{M}\right) \qquad \text{if $-H$ is of type $A$}\\
    \eta_i= -2 + 2\cos \left(\frac{(2i-1)\pi}{M}\right) \qquad \text{if $-H$ is of type $D$}\\
    \end{cases}
\end{equation}
%see eq. (5) in \url{http://www.math.kent.edu/~reichel/publications/toep3.pdf})
the biggest of which is $\eta_1=-2+2\cos(\frac{\pi}{M})$ in both cases. Then
\[ \alpha_1\leq -2+2\cos(\frac{\pi}{M}) + \frac{\sum_{i=1}^M r_i}{N^2} \leq -2+2\cos(\frac{\pi}{N}) + \frac{N-1}{N^2}   \]
since $N\geq M$. Observe that the function $-2+2\cos(\frac{\pi}{x})+ \frac{x-1}{x^2}$ is always negative. Hence, $\alpha_1<0$ for all $N$, and $A$ is therefore negative definite.

For singularities of exceptional type, we have
\begin{equation}
    \begin{cases}
    \eta_1= -2 + 2\cos \left(\frac{1\pi}{12}\right) \qquad \text{for type $E_6$}\\
    \eta_1= -2 + 2\cos \left(\frac{\pi}{18}\right) \qquad \text{for type $E_7$}\\
    \eta_1= -2 + 2\cos \left(\frac{\pi}{30}\right) \qquad \text{for type $E_8$}\\
    \end{cases}
\end{equation}
and 
\begin{equation}
   \frac{\sum_{i=1}^M r_i}{N^2}= \begin{cases}
   \frac{11}{24^2} \qquad \text{for type $E_6$}\\
   \frac{17}{48^2} \qquad \text{for type $E_7$}\\
   \frac{29}{120^2} \qquad \text{for type $E_8$}\\
    \end{cases}
\end{equation}
\AAA{Argue similarly for the E cases, a reference is \url{https://arxiv.org/pdf/1509.05591.pdf} for $E_6 and E_8$, can't seem to find $E_7$... }

\end{comment}

Notice that replacing $Q_0$ with a quadratic form
\[ 
  Q\coloneqq  Q_0 + S(\re Z)^2 + T(\im Z)^2 
\]
does not affect its signature on $\ker Z_{w,\gamma}$.

The following lemma holds on projective orbisurfaces as well (it can be proven
on the coarse moduli space):

\begin{lemma}[{\cite[Ex. 6.11]{MS17}}]\label{lem_ample_class_inequality}
  Let $\omega$ be an ample real divisor class. Then there exists a constant
  $C_\omega\geq 0$ such that, for every effective divisor class $D$, we have 
  \[ 
    C_\omega(\omega . D)^2 + D^2 \geq 0. 
  \]
\end{lemma}

Let $C_H$ the constant in the lemma corresponding to the class $H$, and define
\begin{equation}
  Q_1\coloneqq Q_0 + C_H(\im Z)^2.
\end{equation}

\begin{lemma}\label{lem_Q1_nonneg}
  The form $Q_1$ satisfies $Q_1(E)\geq 0$ if $E$ a torsion-free slope-semistable
  sheaf, or if $E$ is supported on a curve.
\end{lemma}

\begin{proof}
  If $E$ is torsion supported on a curve then $\ch 1(E)$ is effective and 
  \[
    Q_1(E)= \ch 1(E)^2 + C_H(H.\ch 1(E))^2\geq 0
  \]
  where the inequality is that of Lemma \ref{lem_ample_class_inequality}, and
  $Q_1$ reduces to that expression because $\ch 0 (E)=0$ and $\delta(E)=0$ as
  argued in the proof of Lemma \ref{lem_estimate_delta_for_torfree}. 

  If $E$ is a torsion-free, semistable sheaf, we have shown that
  \[ 
    Q_0(E)=\Delta_{orb}(E)\geq 0 
  \]
  by Theorem \ref{thm_discriminant}. 
\end{proof}

Now observe that if $T$ is a sheaf supported on points, we may write \[ [T] =
d[\sO_{Np}] + \sum_i (dr_i+t_i)[\sO_p\otimes\rho_i]\] with $d\geq 0$ and
$t_i+dr_i\geq 0$. Let 
\[
  K\coloneqq \max\limits_{i=1,...,M}\set{Z_{w,\gamma}(\sO_p),Z_{w,\gamma}(\sO_p\otimes\rho_i)}<0.
\]
Then $\re Z_{w,\gamma}(T)^2\geq K^2(d +\sum(d + t_j))^2$. Now pick $S$ such that
$SK^2>2N$, and observe that 
\begin{equation}\label{eq_Q1_positive_on_points}
  \begin{split}
    Q_1(T) + S(\re Z(T))^2 = & \left(\sum t_iC_i\right)^2  + S(\re Z(T))^2\geq \\
    & \left(\sum t_iC_i\right)^2  + SK^2\left(d +\sum(dr_i + t_i)\right)^2\geq\\
    & -2\sum t_j^2  + SK^2\left(d^2 + \sum(dr_i + t_j)^2\right),\\
    & -2\sum t_j^2  + SK^2\left(d^2 + \sum \left(d^2r_i^2 + t_i^2 + 2dr_i t_i\right)\right),
  \end{split}
\end{equation}
where we used that the eigenvalues of the intersection matrix of $C$ are all
$\geq -2$, and that the mixed products appearing in the second summand are all
non-negative. Now write 
\begin{equation}\label{eq_complete_square_on_points}
    \begin{split}
  d^2 + \sum \left(d^2r_i^2 + t_i^2 + 2dr_i t_i\right) =\\
% \frac{\sum t_i^2}{N} + d^2  + \sum \left( -\dfrac{1}{N-1}d^2r_i^2 + \frac{N}{N-1}d^2r_i^2 + 2dr_i t_i + \frac{N-1}{N}t_i^2 \right)=\\
  \frac{\sum t_i^2}{N} + d^2 - d^2\frac{\sum r_i^2}{N-1} +  \sum \left( \frac{N}{N-1}d^2r_i^2 + 2dr_i t_i + \frac{N-1}{N}t_i^2 \right)=\\
    \frac{\sum t_i^2}{N} +  \sum \left( \sqrt{\frac{Ndr_i}{N-1}}   + \sqrt{\frac{(N-1)t_i}{N}} \right)^2 \geq  \frac{\sum t_i^2}{N}.
    \end{split}
\end{equation}
We may continue the chain of inequalities \eqref{eq_Q1_positive_on_points}:
\begin{equation*}
    Q_1(T) + S(\re Z(T))^2 \geq -2\sum t_j^2  + \frac{SK^2}{N}\sum t_j^2 > 0
\end{equation*}
by our choice of $S$. We can finally define 
\[ 
  Q(E)\coloneqq Q_1(E) + S(\re Z(E))^2 
\] 
and observe that $Q$ is negative definite on $\ker Z_{w,\gamma}$ and
non-negative on objects of $\mathcal R$. Then we have:

\begin{thm}[Rational case]
  The pair $\sigma_{w,\gamma}=(Z_{w,\gamma},\Coh^{-\im w}(\sS))$ satisfies the
  support property with respect to the quadratic form $Q$, and it is then a
  stability condition on $D(\sS)$.
\end{thm}

\begin{proof}
  This is a standard argument (see for example \cite[Theor. 6.11]{Tra17} or
  \cite[Theor. 6.13]{MS17}). Suppose $E$ is $\sigma_{w,\gamma}$-semistable and
  $\im Z_{w,\gamma}(E)$ is minimal (note that such a minimum exists only because
  we are tilting at a rational slope, and hence the image of $\im Z_{w,\gamma}$ is
  discrete).  Then $E$ must be semistable for all $\alpha'>\alpha=\re w$, and thus
  belong to $\mathcal R$, so $Q(E)\geq 0$. 

  We proceed now by induction on $\im Z_{w,\gamma}$: suppose there is a
  semistable object $E$ for which $Q(E)<0$, and that $Q(F)\geq 0$ for all
  semistable objects with $\im Z_{w,\gamma}(F)<\im Z_{w,\gamma}(E)$. The object
  $E\notin \mathcal R$, so there is some $\alpha'>\alpha$ such that $E$ is
  strictly semistable with respect to  $\sigma_{w',\gamma}$ (where
  $w'=\alpha'+i\im w$), with Jordan-H\"older factors $E_1,...,E_m$. The
  inductive hypothesis applies to $E_i$, hence $Q(E_i)\geq 0$ for all $i$. The
  images $Z_{w',\gamma}(E_i)$ all lie in the same ray in $\C$, so for any pair
  $E_i$, $E_j$ there exists $a>0$ such that
  $Z_{w',\gamma}(E_i)-aZ_{w',\gamma}(E_j)=0$. Since $Q$ is negative definite on
  $\ker Z_{w',\gamma}$, the class $[E_i]-a[E_j]$ belongs to the negative cone of
  $Q$ in $K(\sS)$. This implies that any linear combination with positive
  coefficients of $[E_i]$ and $[E_j]$ lies in the positive cone of $Q$. Since
  this holds for any $i,j$, we must have $Q(E)\geq 0$. 

  This shows that the support property is satisfied for all semistable objects
  of positive imaginary charge. We checked above that the support property with
  respect to $Q$ is satisfied for stable objects of phase 1 as well, which
  allows us to conlcude.
\end{proof}

\begin{proof}[Proof of Theorem \ref{thm_OurStabilityCondition}]
  It remains to argue that one can drop the rationality assumptions on $H$ and
  $\im w$. The argument is carried out in detail in \cite{Bri08_k3} in the case
  of a K3 surface and follows from the discussion in \cite[Sec.
  6,7]{Bri07_triang_cat}, but it requires an observation about the heart of a
  stability condition in the geometric chamber. The analog of this observation
  is the following Lemma, which can be proven exactly as \cite[Lemma 6.20]{MS17}
\end{proof}

\begin{lemma}
  Let $(\sB,Z_{w,\gamma})$ be a stability condition for which all skyscraper
  sheaves $\sO_p\otimes\rho^i$, $j=0,...,M$ and $\sO_{q}$ for $q\neq p$ are stable of
  phase 1. Then $\sB=\Coh^{-\im w}(\sS)$. 
\end{lemma}

\section{Wall-crossing: clusters and constellations}
\label{sec_WallCrossing}

Throughout this section, $\sS$ is an $ADE$-orbisurface with a single isolated stacky point $p$. We investigate wall-crossing for objects of class $v\coloneqq [\sO_x]$, where $x\in \sS$ is a closed point with trivial stabilizer. Let $\sigma^*\coloneqq\sigma_{w,\gamma}$ be one of the stability conditions of Theorem \ref{thm_OurStabilityCondition}, and denote by $\sA$ its heart $\Coh^{-\im w}(\sS)$. 

\begin{lemma}\label{lem_skyscrapers_are_simple}
Skyscraper sheaves $\sO_x$ for $x\neq p$, and sheaves $\sO_p\otimes \rho$ are simple objects in $\sA$. Therefore, they are $\sigma^*$-stable and all have phase 1. 
\end{lemma}

\begin{proof}
The long exact sequence of cohomology sheaves associated to a short exact sequence in $\sA$
\[ 0\to A\to \sO_x \to B \to 0 \]
shows that $H^0(B)$ is 0 or $\sO_x$. If $H^0(B)\simeq \sO_x$, then \[H^{-1}(B)\simeq H^0(A)=A=0.\]
If $H^0(B)=0$, then $H^{-1}(B)$ and $H^0(A)$ have the same slope, which is a contradiction, unless $H^{-1}(B)=0$ and $A\simeq \sO_x$. This shows that $\sO_x$ is simple in $\sA$. The argument for sheaves $\sO_p\otimes\rho$ is identical. One then observes that $Z_{w,\gamma}$ maps these objects to the negative real axis to conclude they are stable of phase 1. 
\end{proof}

\begin{lemma}
The objects of class $[\sO_x]$ in $\sA$ are skyscraper sheaves $\sO_x$, or they have a composition series whose factors are the $\sO_p\otimes \rho_i$, repeated with multiplicity $r_i=\dim \rho_i$. The former are $\sigma^*$-stable, while the latter are $\sigma^*$-semistable and all share the same $S$-equivalence class. 
\end{lemma}

\begin{proof}
The statement about stability follows immediately from Lemma \ref{lem_skyscrapers_are_simple}. What needs justification is the "only" part of the statement: let $A\in \sA$ be a complex with $[\sO_x]=[A]=[H^0(A)]-[H^{-1}(A)]$. Then $H^0(A)$ and $H^{-1}(A)$ have the same slope, since their classes differ by a codimension 2 summand. This is only possible if $H^{-1}(A)=0$ and $A$ is a sheaf of class $[\sO_x]$. If $A$ is supported away from $p$, then $A=\sO_x$ for some $x$. If $A$ is supported at $p$, then it has a composition series with factors $\sO_p\otimes \rho_i$. The multiplicities must be $r_i$ since $[\sO_x]=\sum_ir_i[\sO_p\otimes\rho_i]$.
\end{proof}

Now apply Bridgeland's deformation result \ref{thm_Deformation_Bridgeland} to $\sigma^*$, and obtain a neighborhood $\Delta\subset \Stab(\sS)$ and a stability condition $\sigma_0=(Z_0,\sA_0)\in \Delta$ such that $Z_0([\sO_x])=-1$ and $\phi_0(\sO_p\otimes \rho)<1$ for all $\rho\neq\mathbbm 1$. Denote by $\mathcal V\simeq [W/G]$, with $W=\C^2$, the chart of $\sS$ around the stacky point, and recall Definition  \ref{def_cluster_constellation}.

\begin{proposition}\label{prop_recoverGHilb}
The moduli space $M_{\sigma^*}(v)$ of $\sigma^*$-semistable objects is isomorphic to the coarse moduli space of $\sS$, and $M_{\sigma_0}(v)$ is isomorphic to its minimal resolution. The wall-crossing morphism 
\[ M_{\sigma_0}(v) \to M_{\sigma^*}(v) \]
sending $\sigma_0$-semistable objects to their $\sigma^*$-S-equivalence class is the contraction of the exceptional divisors.
\end{proposition}

\begin{proof}
The stability condition $\sigma_0$ ensures that no proper subsheaf of a $\sigma_0$-stable object contains $\sO_p\otimes \mathbbm 1$ in its composition series. Then, objects of $M_{\sigma_0}(v)$ supported at $p$ are exactly $G$-clusters, and $M_{\sigma_0}(v)$ is locally isomorphic to $G$-Hilb$(W)$. The $G$-Hilbert scheme is the minimal resolution of $W/G$ \cite{BKR01}, with exceptional locus parameterizing clusters supported at $p$. 

At $\sigma^*$, clusters become strictly semistable with the same $S$-equivalence class. In other words, the exceptional divisor of $G$-Hilb$(W)$ is contracted to a point, showing that $M_{\sigma^*}(v)$ is locally isomorphic to $V\coloneqq W/G$. Away from this chart,  $M_{\sigma^*}(v)$ is isomorphic to $S\setminus V$. This shows that $M_{\sigma^*}(v)\simeq S$ and concludes the proof of the proposition.
\end{proof}

In a completely analogous way, we can define deformations $\sigma_i=(Z_i,\sA_i)$ of $\sigma^*$ such that $Z_i([\sO_x])=-1$ and $\phi_0(\sO_p\otimes \rho)<1$ for all $\rho\neq \rho_i$. One argues then as in Prop. \ref{prop_recoverGHilb}, to show that the moduli spaces $M_{\sigma_i}(v)$ are moduli spaces of $G$-constellations, and are crepant resolutions of $M_{\sigma^*}(v)$. This is not surprising, as the stability conditions $\sigma_i$ correspond to certain generic stability parameters on quiver representations, as illustrated in the next subsection. Proposition \ref{prop_recoverGHilb} is then an analog of a well-known result in King's theory of stability for quiver representations, see \cite{Kin94} and \cite{CR04}.

\subsection{Comparison with \cite{Bri09_kleinian}}
\label{sec_ComparisonBridgeland}

Let $\sB$ be the finite length abelian subcategory of $\Coh(\sS)$ generated by the simple sheaves $\sO_p\otimes \rho_i$, with $i\neq 0$. Denote by $\sT$ the triangulated subcategory of $D(\sS)$ consisting of complexes whose cohomologies lie in $\sB$, the main result of \cite{Bri09_kleinian} is the description of a connected component of $\Stab(\sT)$. 

The Grothendieck group $K(\sT)$ endowed with the Euler pairing is a root lattice and the classes $\alpha_i\coloneqq [\sO_p\otimes\rho_i]$ are roots. Therefore, the space $\Hom(K(\sT),\C)$ of central charges of $\Stab(\sT)$ is identified with the Cartan algebra of the root system, and admits an action of the Weyl group which is free on the set of regular orbits
\[ \mathfrak{h}^{\mathrm{reg}}=\{ Z\in \Hom(K(\sT),\C) \mid Z(\alpha)\neq 0  \mbox{ for all roots }\alpha \in K(\sT)\}. \]
By \cite[Lemma 3.1]{Bri09_kleinian}, there is a region $U$ in $\Stab(\sT)$, homeomorphic to a complexified Weyl chamber in $\mathfrak{h}^{\mathrm{reg}}$, containing stability conditions $(Z,\sB)$ with $\im Z(\sO_p\otimes\rho_i)>0$ for all $i\neq 0$. Moreover, the central charge map $\varpi\colon\Stab(\sT)  \to \Hom(K(\sT),\C)$ is a covering space over $\mathfrak{h}^{\mathrm{reg}}$ \cite[Prop. 3.3]{Bri09_kleinian}.

Since $Z_{0|K(\sT)}$ satisfies $\im Z_{0|K(\sT)}(\sO_p\otimes\rho_i)>0$ for all $i\neq 0$, the stability condition $\sigma_0$ gives rise to a stability condition $(Z_{0|K(\sT)},\sB)\in U$ by \cite[Lemma 3.1]{Bri09_kleinian}. More generally, stability conditions $\sigma=(Z_\sigma,\sA_\sigma) \in \Delta \subset \Stab(D(\sS))$, satisfy $Z_{\sigma|K(\sT)}\in \mathfrak{h}^{\mathrm{reg}}$.

Therefore, restriction of central charge defines a map $\Delta \to \mathfrak{h}^{\mathrm{reg}}$, which lifts to a map 
\[\delta\colon\Delta \to \Stab(\sT).\]

The boundary $\partial U$ decomposes as $\partial U=\cup_i U_i$, where 
\[ U_i\coloneqq \{ \tau\in \partial U \mid \im Z_\tau(\sO_p\otimes\rho_i)=0 \}.\] 
This shows that $\delta(\sigma^*)\in \cap_i U_i$. Moreover, $\delta^{-1}(U_i)$ is a wall for class $v$, because $\sigma\in \delta^{-1}(U_i)$ satisfies $Z_\sigma(v)// Z_\sigma(\sO_p\otimes\rho_i)$. Summarizing:

\begin{proposition}
There exists a map $\delta\colon \Delta\to \Stab(\sT)$. 
The preimages of the components of $\partial U$ along $\delta$ are walls for class $v$ in $\Stab(\sS)$, and $\sigma^*$ lies in their intersection.
\end{proposition}

\bibliographystyle{amsplain}
\bibliography{./bibliography}

\providecommand{\bysame}{\leavevmode\hbox to3em{\hrulefill}\thinspace}
\providecommand{\MR}{\relax\ifhmode\unskip\space\fi MR }
\providecommand{\MRhref}[2]{%
  \href{http://www.ams.org/mathscinet-getitem?mr=#1}{#2}
}
\providecommand{\href}[2]{#2}
\begin{thebibliography}{10}

\bibitem{AB13}
Daniele Arcara and Aaron Bertram, \emph{Bridgeland-stable moduli spaces for
  {$K$}-trivial surfaces}, J. Eur. Math. Soc. (JEMS) \textbf{15} (2013), no.~1,
  1--38, With an appendix by Max Lieblich. \MR{2998828}

\bibitem{BMS16}
Arend Bayer, Emanuele Macr\`\i, and Paolo Stellari, \emph{The space of
  stability conditions on abelian threefolds, and on some {C}alabi-{Y}au
  threefolds}, Invent. Math. \textbf{206} (2016), no.~3, 869--933. \MR{3573975}

\bibitem{BMSZ17}
Marcello Bernardara, Emanuele Macr\`\i, Benjamin Schmidt, and Xiaolei Zhao,
  \emph{Bridgeland stability conditions on {F}ano threefolds}, \'{E}pijournal
  Geom. Alg\'{e}brique \textbf{1} (2017), Art. 2, 24. \MR{3743105}

\bibitem{Bri02_flops}
Tom Bridgeland, \emph{Flops and derived categories}, Invent. Math. \textbf{147}
  (2002), no.~3, 613--632. \MR{1893007}

\bibitem{Bri07_triang_cat}
\bysame, \emph{Stability conditions on triangulated categories}, Ann. of Math.
  (2) \textbf{166} (2007), no.~2, 317--345. \MR{2373143}

\bibitem{Bri08_k3}
\bysame, \emph{Stability conditions on {$K3$} surfaces}, Duke Math. J.
  \textbf{141} (2008), no.~2, 241--291. \MR{2376815}

\bibitem{Bri09_kleinian}
\bysame, \emph{Stability conditions and {K}leinian singularities}, Int. Math.
  Res. Not. IMRN (2009), no.~21, 4142--4157. \MR{2549952}

\bibitem{BKR01}
Tom Bridgeland, Alastair King, and Miles Reid, \emph{The {M}c{K}ay
  correspondence as an equivalence of derived categories}, J. Amer. Math. Soc.
  \textbf{14} (2001), no.~3, 535--554. \MR{1824990}

\bibitem{CT19}
Jiun-Cheng Chen and Hsian-Hua Tseng, \emph{On the bogomolov-miyaoka-yau
  inequality for stacky surfaces}, Taiwanese J. Math. (2019), Advance
  publication.

\bibitem{CR04}
Weimin Chen and Yongbin Ruan, \emph{A new cohomology theory of orbifold}, Comm.
  Math. Phys. \textbf{248} (2004), no.~1, 1--31. \MR{2104605}

\bibitem{Dam11}
Pantelis~A. {Damianou}, \emph{{On the characteristic polynomial of Cartan
  matrices and Chebyshev polynomials}}, arXiv e-prints (2011), arXiv:1110.6620.

\bibitem{Dou02}
Michael~R. Douglas, \emph{Dirichlet branes, homological mirror symmetry, and
  stability}, Proceedings of the {I}nternational {C}ongress of
  {M}athematicians, {V}ol. {III} ({B}eijing, 2002), Higher Ed. Press, Beijing,
  2002, pp.~395--408. \MR{1957548}

\bibitem{FMN10}
Barbara Fantechi, Etienne Mann, and Fabio Nironi, \emph{Smooth toric
  {D}eligne-{M}umford stacks}, J. Reine Angew. Math. \textbf{648} (2010),
  201--244. \MR{2774310}

\bibitem{HL10}
Daniel Huybrechts and Manfred Lehn, \emph{The geometry of moduli spaces of
  sheaves}, second ed., Cambridge Mathematical Library, Cambridge University
  Press, Cambridge, 2010. \MR{2665168}

\bibitem{Kin94}
Alastair~D. King, \emph{Moduli of representations of finite-dimensional
  algebras}, Quart. J. Math. Oxford Ser. (2) \textbf{45} (1994), no.~180,
  515--530. \MR{1315461}

\bibitem{KS08}
Maxim {Kontsevich} and Yan {Soibelman}, \emph{{Stability structures, motivic
  Donaldson-Thomas invariants and cluster transformations}}, arXiv e-prints
  (2008), arXiv:0811.2435.

\bibitem{Kuz14}
Alexander Kuznetsov, \emph{Semiorthogonal decompositions in algebraic
  geometry}, Proceedings of the {I}nternational {C}ongress of
  {M}athematicians---{S}eoul 2014. {V}ol. {II}, Kyung Moon Sa, Seoul, 2014,
  pp.~635--660. \MR{3728631}

\bibitem{Li19}
Chunyi Li, \emph{On stability conditions for the quintic threefold}, Invent.
  Math. \textbf{218} (2019), no.~1, 301--340. \MR{3994590}

\bibitem{Lie11}
Max Lieblich, \emph{Moduli of twisted orbifold sheaves}, Adv. Math.
  \textbf{226} (2011), no.~5, 4145--4182. \MR{2770444}

\bibitem{Mac07}
Emanuele Macr\`\i, \emph{Stability conditions on curves}, Math. Res. Lett.
  \textbf{14} (2007), no.~4, 657--672. \MR{2335991}

\bibitem{MS17}
Emanuele Macr\`\i and Benjamin Schmidt, \emph{Lectures on {B}ridgeland
  stability}, Moduli of curves, Lect. Notes Unione Mat. Ital., vol.~21,
  Springer, Cham, 2017, pp.~139--211. \MR{3729077}

\bibitem{Tod13}
Yukinobu Toda, \emph{Stability conditions and extremal contractions}, Math.
  Ann. \textbf{357} (2013), no.~2, 631--685. \MR{3096520}

\bibitem{Toe99}
Bertrand To\"{e}n, \emph{Th\'{e}or\`emes de {R}iemann-{R}och pour les champs de
  {D}eligne-{M}umford}, $K$-Theory \textbf{18} (1999), no.~1, 33--76.
  \MR{1710187}

\bibitem{Tra17}
Rebecca {Tramel} and Bingyu {Xia}, \emph{{Bridgeland stability conditions on
  surfaces with curves of negative self-intersection}}, arXiv e-prints (2017),
  arXiv:1702.06252.

\bibitem{Tse10}
Hsian-Hua Tseng, \emph{Orbifold quantum {R}iemann-{R}och, {L}efschetz and
  {S}erre}, Geom. Topol. \textbf{14} (2010), no.~1, 1--81. \MR{2578300}

\bibitem{VdB04}
Michel Van~den Bergh, \emph{Three-dimensional flops and noncommutative rings},
  Duke Math. J. \textbf{122} (2004), no.~3, 423--455. \MR{2057015}

\bibitem{Vistoli}
Angelo Vistoli, \emph{Intersection theory on algebraic stacks and on their
  moduli spaces}, Invent. Math. \textbf{97} (1989), no.~3, 613--670.
  \MR{1005008}

\end{thebibliography}

\end{document}